\documentclass[11pt]{article}
\title{Reduced Incidence algebras description of cobweb posets and KoDAGs}
\author{Ewa Krot-Sieniawska \\
\\Institute of Computer Science, Bia{\l}ystok University\\
PL-15-887 Bia{\l}ystok, ul.Sosnowa 64, POLAND\\
e-mail: ewakrot@wp.pl, ewakrot@ii.uwb.edu.pl}
\date{}

\usepackage{amsmath,amsthm}
\usepackage{amssymb}
\usepackage{graphicx}
\chardef\bslash=`\\ 
\newtheorem{defn}{Definition}[]

\newtheorem{thm}{Theorem}[]
\newtheorem{prop}{Proposition}[]
\newtheorem{cor}{Corollary}[]

\newcommand{\inc}[2]{\left\langle#1 \atop #2\right\rangle}
\newcommand{\naturals}{{\mathbb N}}

\begin{document}
\maketitle
\begin{abstract}
\noindent The notion of reduced incidence algebra of an arbitrary
cobweb poset is delivered.

\end{abstract}
 \small{KEY WORDS:
 cobweb poset,  incidence algebra of locally finite poset, order compatible equivalence relation, reduced incidence algebra .}\\
AMS Classification numbers: 06A06,  06A07, 06A11,  11C08, 11B37\\

\noindent Presented at Gian-Carlo Rota Polish Seminar:
http://ii.uwb.edu.pl/akk/sem/sem rota.htm
\section{ Cobweb posets}
The   family of the so called cobweb posets $\Pi$ has been
invented by A.K.Kwa\'sniewski few years ago (for references  see:
\cite{44,46}). These structures are such a generalization of the
Fibonacci tree growth  that allows joint combinatorial
interpretation for all of them under the admissibility condition
(see \cite{49,49a}).

\noindent Let $\{F_n\}_{n\geq 0}$ be a natural numbers valued
sequence with $F_0=1$ (with $F_0=0$ being exceptional as in case
of Fibonacci numbers). Any sequence satisfying this property
uniquely designates cobweb poset  defined as follows.

\noindent For $s\in\bf{N}_0=\bf{N}\cup\{0\}$ let us to define
levels of
 $\Pi$:
$$\Phi_{s}=\left\{\langle j,s \rangle ,\;\;1\leq j \leq F_{s}\right\},\;\;\;$$
(in case of $F_0=0$  level $\Phi_0$ corresponds to the empty root
$\{\emptyset\}$). )

\noindent Then

\begin{defn}
Corresponding cobweb poset is  an infinite partially ordered set
$\Pi=(V,\leq)$, where
 $$ V=\bigcup_{0\leq s}\Phi_s$$
 are the elements ( vertices) of $\Pi$ and the partial order relation $\leq$ on $V$ for
 $x=\langle s,t\rangle, y=\langle u,v\rangle $ being  elements of
cobweb poset $\Pi$ is defined by  formula
$$ ( x \leq_{P} y) \Longleftrightarrow
 [(t<v)\vee (t=v \wedge s=u)].$$
\end{defn}

\noindent Obviously any cobweb poset can be represented, via its
Hasse diagram, as infinite directed  graf  $\Pi=\left(
V,E\right)$, where  set $V$ of its vertices is defined as above
and

$$E =\{\left(\langle j , p\rangle,\langle q ,(p+1) \rangle
\right)\}\;\cup\;\{\left(\langle 1 , 0\rangle ,\langle 1 ,1
\rangle \right)\},$$ \quad where $1 \leq j \leq {F_p}$ and $1\leq
q \leq {F_{(p+1)}}$ stays for  set of (directed) edges.

\noindent The Kwa\'sniewski cobweb posets under consideration
represented by  graphs  are examples of oderable directed acyclic
graphs (oDAG)  which we start to call from now in brief:  KoDAGs.
These are  structures of universal importance for the whole of
mathematics - in particular for discrete "`mathemagics"'
[http://ii.uwb.edu.pl/akk/ ]  and computer sciences in general
(quotation from \cite{49,49a} ):

\begin{quote}
For any given natural numbers valued sequence the graded (layered)
cobweb posets` DAGs  are equivalently representations of a chain
of binary relations. Every relation of the cobweb poset chain is
biunivocally represented by the uniquely designated
\textbf{complete} bipartite digraph-a digraph which is a
di-biclique  designated  by the very  given sequence. The cobweb
poset is then to be identified with a chain of di-bicliques i.e.
by definition - a chain of complete bipartite one direction
digraphs.   Any chain of relations is therefore obtainable from
the cobweb poset chain of complete relations  via
deleting  arcs (arrows) in di-bicliques.\\
\end{quote}

\noindent According to the  definition above arbitrary cobweb
poset $\Pi=(V,\leq)$ is a graded poset (ranked poset) and for
$s\in\bf
 N_0$:
 $$x\in\Phi_s\;\; \longrightarrow\;\; r(x)=s,$$
 where $r:\Pi\rightarrow \bf N_0$ is a rank function on $\Pi$.

\noindent Let us then define Kwa\'sniewski finite cobweb
sub-posets as follows
 \begin{defn}
 Let $P_n=(V_n,\leq)$, $(n\geq 0)$, for ${\displaystyle
V_n=\bigcup_{0\leq s\leq n}}\Phi_s$ and  $\leq$ being the induced
partial order relation on $\Pi$.
\end{defn}

\noindent Its easy to see that $P_n$ is ranked poset with rank
function $r$ as above. $P_n$ has a unique minimal element
$0=\langle 1,0\rangle$ ( with $r(0)=0$). Moreover $\Pi$ and all
$P_n$ s are locally finite, i.e. for any pair $x,y\in \Pi$, the
segment $[x,y]=\{z\in\Pi:\,x\leq z\leq y\}$ is finite.

Let us recall that one defines the incidence algebra of a locally
finite partially ordered set $P$ as follows (see \cite{65, 73,
74}):
$$ { I(P)}=I(P,R)=\{f: P\times  P\longrightarrow  R;\;\;\;\; f(x,y)=0\;\;\;unless\;\;\; x\leq y\}.$$
The sum of two such functions $f$ and $g$ and multiplication by
scalar are defined as usual. The product $h=f\ast g$ is defined as
follows:
$$ h(x,y)=(f\ast g)(x,y)=\sum_{z\in {\bf P}:\;x\leq z\leq y} f(x,z)\cdot g(z,y).$$
It is immediately verified that this is an associative algebra
(with  an identity element $\delta (x,y)$, the Kronecker delta),
over  any associative ring  R.

In \cite{37b} the incidence algebra of an arbitrary cobweb poset
$\Pi$ ( or its subposets $P_n$) uniquely designated
 by the natural numbers valued sequence $\{F_n\}_{n\geq 0}$, was considered by the present author.
The  explicit formulas for some typical elements  of incidence
algebra   ${I(\Pi)}$ of $\Pi$ where delivered there.

So for $x,y$ being some arbitrary elements of $\Pi$ such that
$x=\langle s,t\rangle$, $y=\langle u,v\rangle $, $(s,u\in\bf{N}$,
$t,v\in \bf{N}_0)$, $1\leq s\leq F_t$   and $1\leq u\leq F_v$ one
has:

\renewcommand{\labelenumi}{(\arabic{enumi})}
\begin{enumerate}
\item $\zeta$ function of ${\Pi}$ being a characteristic
 function of partial order in $\Pi$
  \begin{equation}
\zeta(x,y)=\zeta\left( \langle s,t\rangle ,\langle u,v \rangle
\right)=\delta (s,u)\delta (t,v)+\sum_{k=1}^{\infty}\delta(t+k,v),
\end{equation}
one can also verify, that $\zeta^k$ enumerates all multichains of
length $k$, \item M\"{o}bius function of $\Pi$ being a inverse of
$\zeta$
\begin{equation}\label{mobius2}
\mu(x,y)=\delta(t,v)\delta(s,u)-\delta(t+1,v)+\sum_{k=2}^{\infty}\delta(t+k,v)(-1)^{k}
\prod_{i=t+1}^{v-1}(F_{i}-1),
\end{equation}
\item function $\zeta ^{2}=\zeta \ast \zeta$ counting the number
of elements in the segment $\left[ x,y\right]$

\begin{equation}
\begin{array}{lll}
\zeta^{2}(x,y)=  card \left[ x, y\right] & = &
 \left({\displaystyle\sum_{i=t+1}^{v-1}}F_{i}\right)+2 ,\\
\end{array}\end{equation}
\item function $\eta$
\begin{equation}
\eta(x,y)=\sum_{k=1}^{\infty}\delta(t+k,v)=\left\{\begin{array}{ccc}
  1 &  & t<v \\
  0 &  & w\;p.p. \\
\end{array}\right.,\end{equation}
\item function $\eta ^{k}(x,y),\;\; (k\in {\bf N})$ counting the
number of chains of length $k$, (with $(k+1)$ elements) from $x$
to $y$
\begin{equation}
\begin{array}{lll}
   \eta^k(x,y) & = & {\displaystyle \sum_{x< z_1<z_2<...<z_{k-1}<
   y}}1\\&&\\
   & =&{\displaystyle \sum_{t<i_1<i_2<...<i_{k-1}<v}} F_{i_1}F_{i_2}...F_{i_{k-1}},\\
 \end{array}\end{equation}
 \item function $\mathcal{C}$
 \begin{equation}\mathcal{C}(\langle s,t\rangle , \langle
u,v\rangle)=
\delta(t,v)\delta(s,u)-\sum_{k=1}^{\infty}\delta(t+k,v),\end{equation}
such that its inverse function $\mathcal{C}^{-1}(x,y)$ counts the
number of all chains from $x$ to $y$' \item function $\chi$
\begin{equation} \chi(x,y)=\delta(t+1,v),\end{equation}
\item function $\chi ^{k}(x,y),\;\; (k\in {\bf N})$ counting the
number of maximal chains of length $k$, (with $(k+1)$ elements)
from $x$ to $y$
\begin{equation}
\chi^k(x,y)=\sum_{x\lessdot z_1\lessdot ...\lessdot
z_{k-1}\lessdot y}1=\delta (t+k,v)F_{t+1}F_{t+2}...F_{v-1}.
\end{equation}
\item function $\mathcal{M}$
\begin{equation}\mathcal{M}(\langle s,t\rangle ,\langle
u,v\rangle)=\delta(t,v)\delta(s,u)-\delta(t+1,v),\end{equation}
such that its inverse function  $\mathcal{M}^{-1}$ counts the
number of all maximal chains from $x$ to $y$.

\end{enumerate}

In this paper the notion of the standard reduced incidence algebra
\cite{24,73,74} of  an arbitrary cobweb poset $\Pi$ will be
delivered. As we shall see,
 it enables us  for example to facilitate the formulas presented
 above. The results presented below stay true when considering
 finite subposets $P_n$ defined above.

 \section{The Standard Reduced  Incidence Algebra of an arbitrary
 cobweb poset}

 Let $S(\Pi)$  be the set of all segments in $\Pi$ and let $\sim$ be the equivalence relation  $\sim \subseteq S(\Pi)\times
S(\Pi)$

Let us recall that $\sim$ is compatible (\cite{24}), i.e. it
satisfies the following condition:  if $f$ and $g$ belong to the
incidence algebra $I(P)$ and $f(x,y)=f(u,v)$ as well as
$g(x,y)=g(u,v)$ for all pairs of segments such that $[x,y]\sim
[u,v]$, then $(f\ast g)(x,y)=(f\ast g)(u,v)$.

 The equivalence classes of segments of $\Pi$ relative to $\sim$ are called
 types.
 The set of all functions defined on types (i.e. all functions taking the same value on equivalent
 segments) forms an associative algebra with identity. One calls
 it the reduced incidence algebra $R(\Pi,\sim)$ (modulo the the equivalence relation
 $\sim$). Let us note that $R(\Pi,\sim)$ is isomorphic to a
 subalgebra of the $I(\Pi)$, (\cite{24}).

  Now let be $\sim$ defined as follows

\begin{equation}  \label{sim} [x,y]\sim [u,v]
\Longleftrightarrow\; segments \; [x,y],
 [u,v]\; are\; isomorphic .\end{equation}

 One can show that it is order compatible. Then one calls
 $R(\Pi,\sim)=R(\Pi)$ the standard reduced incidence algebra of
 $\Pi$. Also from the definitions of  $\sim$  and partial order on
 $\Pi$  one infers that
$$[x,y]\sim [u,v]\Longleftrightarrow \left[ r(x)=r(u)\wedge r(y)=r(v)\right].$$
So let $T$ be the set of types of relation $\sim$ defined above.
Then
$$T=\{(k,\,n):\;k,n\in\bf{N}_0\}$$
 and for $k\leq n$
\begin{equation}
(k,\,n)=\{[x,y]\in S(\Pi)\,:\;r(x)=k \wedge\,r(y)=n\},
\end{equation}
 or equivalently
$$(k,\,n)=\{[x,y]\in S(\Pi)\,:\;x\in\Phi_k \wedge\,y\in\Phi_n\}.$$
Also let $(k,n)=\emptyset$ for $k>n.$
\begin{defn}
Let $[x,y]\in\alpha_{k,\,n}$. For $l\in\naturals_0$  one can
define the incidence coefficients in $R(\Pi)$ as follows
\begin{equation}
\inc{k,\,n}{l}=\left|\left\{z\in [x,y]: [x,z]\in(k,\,l)\wedge
[z,y]\in (l,\,n)\right\}\right|.
\end{equation}
\end{defn}
The the following formula holds.
\begin{prop}
\begin{equation}
\inc{k,\,n}{l}=\left\{\begin{array}{ccc}
  F_l &  & k\leq l\leq n \\
  0 &  & othervise \\
\end{array}\right.
\end{equation}\end{prop}

Now one can define the product $\ast$ in $R(\Pi)$ as follows.
\begin{prop} Let  $[x,y]\in (k,\,n)$, ($k \leq n$) and
$f,g\in R(\Pi)$. Then
\begin{equation}
(f\ast g)(k,\,n)={\displaystyle \sum_{ l\geq
0}}F_lf(k,\,l)g(l,\,n)
\end{equation}
with the assumption that $(k,\,n)=\emptyset$, for $l<k$ or
$l>n$.\end{prop}
\begin{proof}\textrm{}\\
$
\begin{array}{lll}
  (f\ast g)({k,\,n}) & = & (f\ast g)(x,y)={\displaystyle \sum_{x\leq z\leq y}}f(x,z)g(z,y)
  \\&&\\
   & = & {\displaystyle \sum_{k\leq l\leq n}\sum_{\{z:[x,z]\in(k,\,l),\;[z,y]\in(l,\,n)\}}}f(x,z)g(z,y)
   \\&&\\
   & = & {\displaystyle \sum_{k\leq l\leq n}}F_lf({k,\,l})g({l,\,n})
   \\&&\\
    & = & {\displaystyle \sum_{ l\geq 0}}F_lf({k,\,l})g({l,\,n}), \\
\end{array}$
\end{proof}
So we have proved

\begin{thm}
Let  $F=\{F_n\}_{n\geq 0}$ z $F_0=1$, be an arbitrary natural
numbers valued sequence with $F_0=1$(with $F_0=0$ being
exceptional). The the numbers  $F_n$ ($n\geq 0$)  are the
incidence coefficients in the standard reduced algebra $R(\Pi)$ of
cobweb poset $\Pi$ uniquely designated by the sequence
$F=\{F_n\}_{n\geq 0}$.
 \end{thm}

One can also show the following
\begin{thm} Let $f\in R(\Pi)$. Then for  $x,y\in \Pi$ such that $[x,y]\in(k,\,n)$ the value $f(x,y)$ depends on
$r(x)=k$ and $r(y)=n$) only, i.e.
\begin{equation}
f(x,y)=f(k,n)=f(r(x),r(y)).
\end{equation}
From the definition of partial order on  $\Pi$ one can also infer
that for $x,y\in \Pi$ satisfying $r(x)=r(y)$ and for $f\in
I(\Pi)$,  one has
$$f(x,y)=\delta(x,y).$$
\end{thm}

It is known that all elements of $I(\Pi)$ mentioned above, i.e.
functions:
  $\zeta$, $\mu$, $\zeta ^{2}$, $\zeta ^{k}$, $\eta$, $\eta^k$,
  $\mathcal{C}$, $\mathcal{C}^{-1}$, $\chi$, $\chi^k$,
  $\mathcal{M}$, $\mathcal{M}^{-1}$ are the elements of an
  arbitrary reduced incidence algebra $R(\Pi,\sim)$, (i.e. modulo an arbitrary order compatible equivalence relatione $\sim$ on
  $S(\Pi)$). Then the next results follows immediately from this fact and above theorems.

\begin{cor}
Let  $ (k,n)\in T$.  Then:
\begin{equation}
\zeta(k,n)=\left\{\begin{array}{lll}
  1 &  & k\leq n \\
  0 &  & k>n \\
\end{array}\right.;
\end{equation}

 \begin{equation}\zeta^2(k,n)=\left\{\begin{array}{lll}
   \left({\displaystyle
\sum_{i=k}^{n-1}}F_i\right)+2 &  & k\leq n \\
  0 &  & k>n \\
 \end{array}\right.;\end{equation}
\begin{equation}\eta(k,n)=\delta_{k<n}=\left\{\begin{array}{lll}
  1 &  & k<n \\
  0 &  & k\geq n ;\\
\end{array}\right.\end{equation}
\begin{equation}\eta^2(k,n)=\left\{\begin{array}{lll}
   {\displaystyle
\sum_{i=k+1}^{n-1}}F_i &  & k\leq n \\
  0 &  & k>n \\
 \end{array}\right.;\end{equation}
\begin{equation}\eta^s(k,n)=\left\{\begin{array}{lll}
  {\displaystyle
\sum_{k<i_1<...<i_{s-1}<n}}F_{i_1}F_{i_2}...F_{i_{k-1}}  &  & k\leq n \\
  0 &  & k>n \\
 \end{array}\right.;\end{equation}
\begin{equation}\mathcal{C}(k,n)=\left\{\begin{array}{lll}
   1&  & k=n \\
  -1 &  & k<n \\
  0 &  & k>n \\
\end{array}\right.;\end{equation}
\begin{equation}\chi(k,n)=\delta(k+1,n);\end{equation}
 \begin{equation}\chi^s(k,n)=\delta(k+s,n)\cdot
F_{k+1}F_{k+2}...F_{n-1};\end{equation}

\begin{equation}\mu(k,n)=\left\{\begin{array}{lll}
  (-1)^{n-k}{\displaystyle \prod_{i=k+1}^{n-1}}(F_i-1) &  & k\leq n \\
    0 &  & k>n  \\
\end{array}\right..\end{equation}
\end{cor}
\begin{cor}
The standard reduced incidence algebra $R(\Pi)$ is the maximally
reduced incidence algebra $\overline{R}(\Pi)$, i.e. the smallest
reduced incidence algebra on $\Pi$. Equivalently the equivalence
relation defined by (\ref{sim}) is the maximal element in the
lattice  of all order compatible equivalence relations on
$S(\Pi)$.

\end{cor}

\noindent {\bf Acknowledgements}\\
Discussions with Participants of Gian-Carlo Rota Polish Seminar,\\
http://ii.uwb.edu.pl/akk/sem/sem\_rota.htm are highly appreciated.

 \end{document}